\documentclass[12pt]{article}
\usepackage{epsfig}
\usepackage{amsfonts}
\usepackage{amsthm}
\usepackage{amsmath}
\usepackage[bottom=2.5cm, left=3cm, right=3cm]{geometry}

\usepackage{graphicx}  
\usepackage{url}  
\newtheorem{theorem}{Theorem}
\newtheorem{lemma}[theorem]{Lemma}

\newtheorem{prop}{Proposition}
\newtheorem{Que}{Question}
\newtheorem{cor}{Corollary}

\begin{document}
\title{Euler's idoneal numbers and an inequality concerning minimal graphs with a prescribed number of spanning trees} 
\date{\today}
\author{Jernej Azarija \; \; \;Riste \v{S}krekovski\\ Department of Mathematics,\\ University of Ljubljana,\\ Jadranska 21,\\ 1000 Ljubljana,\\ Slovenia\\ \texttt{ \{jernej.azarija,skrekovski\}@gmail.com}}

\maketitle
\begin{abstract}

Let $\alpha(n)$ be the least number $k$ for which there exists a simple graph with $k$ vertices having precisely $n \geq 3$ spanning trees. Similarly, define $\beta(n)$ as the least number $k$ for which there exists a simple graph with $k$ edges having precisely $n \geq 3$ spanning trees. As an $n$-cycle has exactly $n$ spanning trees, it follows that $\alpha(n),\beta(n) \leq n$. In this paper, we show that $\alpha(n) \leq  \frac{n+4}{3}$ and $\beta(n) \leq \frac{n+7}{3} $ if and only if $n \notin \{3,4,5,6,7,9,10,13,18,22\}$, which is a subset of Euler's idoneal numbers. Moreover, if $n  \not \equiv 2 \pmod{3}$ and $n \not = 25$ we show that $\alpha(n) \leq \frac{n+9}{4}$ and $\beta(n) \leq \frac{n+13}{4}.$ This improves some previously estabilished bounds.
\end{abstract}


\noindent {\bf AMS Subj. Class. (2010)}:  05C35,  05C05

\section{Introduction}
Results related to the problem of counting spanning trees for a graph date back to 1847. In \cite{Kirchof}, Kirchhoff showed that the number of spanning trees of a graph $G$ is closely related to the cofactor of the Laplacian matrix of $G$.
Later, a number of related results followed. In 1889, Cayley \cite{Cayley} derived the number of spanning trees for the complete graph on $n$ vertices which is $n^{n-2}$. Later formulas for various families of graphs have been derived. For example, it was shown by Baron et al. \cite{Baron} that the number of spanning trees of the square of a cycle $C_n^2$ equals to $nF_n$ where $F_n$ is the $n$'th Fibonacci number. 

Speaking about a seemingly unrelated branch of mathematics, Euler studied around 1778 a special class of numbers allowing him to find large primes. He called such numbers {\em idoneal numbers} ({\em numerus idoneus}). He was able to find 65 such numbers: \begingroup \catcode`,\active \def,{\char`\,\penalty0 } {\em $I$} $= \{1,2,3,4,5,6,7,8,9,10,12,13,15,16,18,21,22,24,25,28,30,33,37,40,42,45,48,57,58,60,70,72,78,85,88,93,102,105,112,120,130,133,165,168,177,190,210,232,240,253,273,280,312,330,345,357,385,408,462,520,760,840,1320,1365,1848\},$ see also \cite{Weil}.  \endgroup 

Gauss \cite{Gauss} conjectured that the set of idoneal numbers {\em I} is complete. It was later proved by Chowla \cite{Chowla} that the set of idoneal numbers is finite. We denote by {\em $I^*$} the set of idoneal numbers not present in {\em $I$} and remark that if the Generalized Riemann Hypothesis is true, then {\em $I^*$} = $\emptyset$ \cite{Weinberger}. It is also known that any idoneal number in {\em $I^*$} has at least six odd prime factors \cite{Pete}. 
In this paper we use the definition of idoneal numbers stating that $n$ is {\em idoneal} if and only if $n$ is not expressible as $n = ab+ac+bc$ for integers $0 < a < b < c$. For other characterizations of idoneal numbers see \cite{Wolfram}. 

We use this number theoretical result to improve the answer related to the question Sedl\'{a}\v{c}ek \cite{Sedlacek} posed in 1970: Given a number $n \geq 3,$ what is the least number $k$ such that there exists a graph on $k$ vertices having precisely $n$ spanning trees? Sedl\'{a}\v{c}ek denoted this function by $\alpha(n)$. He was able to show that $\alpha(n) \leq \frac{n+6}{3}$ for almost all numbers. More precisely he proved that $\alpha(n) \leq \frac{n}{3}+2$ whenever $n \equiv 0 \pmod{3}$ and $\alpha(n) \leq \frac{n+4}{3}$ whenever $n \equiv 2 \pmod{3}$.  
Nebesk\'{y} \cite{Nebresky} later showed that the only fixed points of $\alpha(n)$ are $3,4,5,6,7,10,13$ and $22$, i.e. these are the only numbers $n$ such that $\alpha(n) = n.$ He also defined the function $\beta(n)$ as the least number of edges $l$ for which there exists a graph with $l$ edges and with precisely $n$ spanning trees. He showed that $$ \alpha(n) < \beta(n) \leq \frac{n+1}{2},$$ except for the fixed points of $\alpha$ in which case it holds that $\alpha(n) = \beta(n) = n$. Moreover, as it is observed in \cite{Nebresky}, from the construction used by Sedl\'{a}\v{c}ek \cite{Sedlacek} we have 

$$
\alpha(n) < \beta(n) \leq \left\{
									\begin{array}{ll}
									\frac{n+9}{3} & \mbox{if } n \equiv 0 \pmod{3} \\
									\frac{n+7}{3} & \mbox{if } n \equiv 2 \pmod{3}, 
									\end{array}
\right.	
$$ whenever $n \notin \{3,4,5,6,7,10,13,22\}.$ 

In this paper we improve their result by showing that $$ \beta(n) \leq \frac{n+13}{4}\; , $$ whenever $n \not \equiv 2 \pmod{3}$ and $n \notin \{3,4,6,7,9,13,18,25\}.$ We also prove that $\alpha(22) = \beta(22) = 22$, proof of which in \cite{Nebresky} we found to be incomplete as it only states that there is no graph with cyclomatic number 2 or 3 that has 22 spanning trees and that every graph with a greater cyclomatic number has more than 22 spanning trees.

We will refer to the number of spanning trees of a graph $G$ by $\tau(G)$. Throughout the paper we will often use the following identity used to compute $\tau(G)$: \begin{equation}\label{rec} \tau(G) = \tau(G-e) + \tau(G/e) \end{equation} for every $e \in E(G)$. Here $G/e$ denotes the graph obtained from $G$ by contracting the edge $e$ of $G$ and removing the loop that could possibly be created. Note that the resulting graph may not be simple. If by $G+e$ we denote the graph that is obtained after introducing an edge into $G$ and by $G+P_k$ we denote the graph obtained after interconnecting two vertices of $G$ with a path $P_{k+1}$ of length $k$, then we will occasionally use the fact that: \begin{equation}\label{inequ} \tau(G+e) \geq \tau(G)+2 \quad \hbox{and} \quad \tau(G+P_k) \geq k\,\tau(G)\end{equation} for a connected graph $G$ and $k \geq 2$. The first inequality follows from the fact that we can form at least two spanning trees in $G+e$ that are not spanning trees in $G$ by taking a spanning tree $T$ of $G$ and obtain new trees $T_1$, $T_2$ after removing an edge (not equal to $e$) from the cycle that is obtained in $T+e.$ The second inequality is equally easy to prove.

Graphs with $\alpha(n)$ vertices with $n$ spanning trees possess some structure. For example, it follows directly from equation (\ref{rec}) that graphs having $n$ spanning trees with $\alpha(n)$ vertices are always $2$-edge-connected. A simple argument can then be used to show that such graphs have cycles of length at most $\frac{n}{2}$ provided that $\alpha(n) < n.$

For nonnegative integers $a,b,c$, let $\Theta_{a,b,c}$ be the graph comprised of two vertices connected by three internally disjoint paths of length $a,b$ and $c$, respectively. We refer to these paths as $P_a$, $P_b$ and $P_c$. Note that $\Theta_{a,b,c}$ is simple if and only if at most one of $a,b,c$ equals 1. For $a,b \geq 3$ denote by $C_{a,b}$ the graph obtained after identifying a vertex of an $a$-cycle with a vertex of a disjoint $b$-cycle. Notice that  $\Theta_{a,b,0}$ is isomorphic to $C_{a,b}.$

\section{Lower bounds for the number of spanning trees of graphs derived from $\Theta_{a,b,c}$}

In this section, we examine the number of spanning trees that arise in $\Theta_{a,b,c}$ when interconnecting two distinct vertices by a disjoint path of length $d$. In order to do so we define simple graphs $\Theta_{a,b,c,d}^0(a_1,a_2)$, $\Theta_{a,b,c,d}^1(a_1)$, $\Theta_{a,b,c,d}^2(a_1,b_1)$ that are obtained from $\Theta_{a,b,c}$ by introducing a path. Let $u,v$ be the 3-vertices of $\Theta_{a,b,c}$. 

First we construct $\Theta^0$. We assume $a \geq 3$. For integers $a_1 \geq 1$ and $a_2 \geq 1$ with $a_1 + a_2 < a$, let $x$ and $y$ be the vertices of $P_a$ such that $d_{P_a}(u,x) = a_1$ and $d_{P_a}(v,y) = a_2$. Then $\Theta_{a,b,c,d}^0(a_1,a_2)$ is the graph obtained by interconnecting $x$ and $y$ with a disjoint path of length $d$, see the first graph of Figure \ref{PIC}. As we are only dealing with simple graphs we require that $d > 1$ if $a_1 + a_2 = a-1$. 

We now construct $\Theta^1$. Let $x = u$ and let $y$ be a vertex on $P_a$ such that $d_{P_a}(u,y) = a_1 \geq 1$. Then $\Theta_{a,b,c,d}^1(a_1)$ is the graph obtained by interconnecting $x$ and $y$ by a disjoint path of length  $d$. See the second graph of Figure \ref{PIC}. Notice that the possibility $y = v$ is not excluded; in that case $\Theta_{a,b,c,d}^1(a_1)$ has two 4-vertices. Since we wish that the resulting graph be simple we require $d \geq 2$ whenever $a_1 = 1$ and $d \geq 2$ whenever $\min(a,b,c) = 1$. Moreover, we assume $a \geq 2$. 

Finally, define $\Theta_{a,b,c,d}^2(a_1,b_1)$ by choosing a vertex $x$ on $P_a$, $x \notin \{u,v\}$ such that $d_{P_a}(u,x) = a_1$ and similarly let $y$ be a vertex on $P_b$, $y \notin \{u,v\}$ such that $d_{P_b}(u,y) = b_1$. We always assume $a_1 \geq 1$ and $b_1 \geq 1.$ Then $\Theta_{a,b,c,d}^2(a_1,b_1)$ is the graph obtained by connecting $x$ and $y$ by a disjoint path of length $d$ as shown on the third graph of Figure \ref{PIC}. Observe that this construction requires $a,b \geq 2$, $a-a_1 \geq 1,$ and $b-b_1 \geq 1$ in order to preserve simplicity of $\Theta^2.$

We now present some formulas and inequalities for the number of spanning trees for the graphs we have just defined. The following equalities are easy to derive using equation (\ref{rec}). 

\begin{figure}
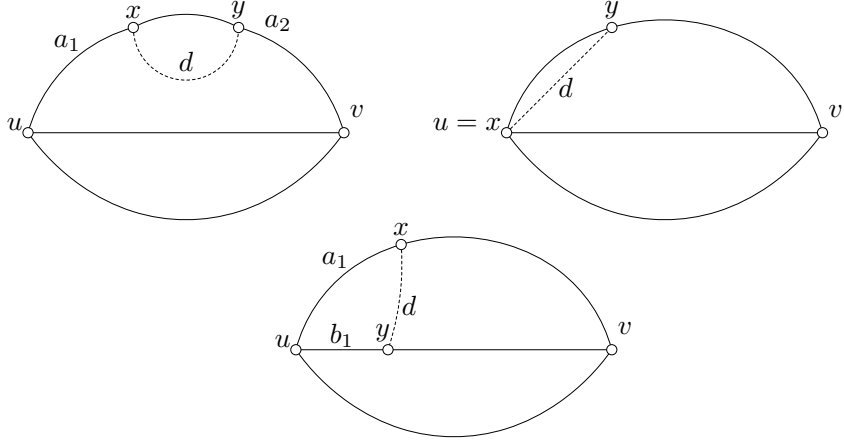

\begin{center}
\includegraphics[scale=0.35]{blah.2} \hbox{   \;\;\;  }
\includegraphics[scale=0.35]{blah.3} \hbox{    \;\;\; } 
\includegraphics[scale=0.35]{blah.4}
\end{center}
\caption{Graphs $\Theta_{a,b,c,d}^0(a_1,a_2), \Theta_{a,b,c,d}^1(a_1), \Theta_{a,b,c,d}^2(a_1,b_1)$ obtained after adding a path of length $d$ between two distinct vertices $x,y$ of $\Theta_{a,b,c}$.}\label{PIC}
\end{figure}

\begin{lemma}\label{LM0}
The following three equalities hold:
\begin{enumerate}
\item[$(1).$] $\tau(\Theta_{a,b,c,d}^0(a_1,a_2))  =  d\,\tau(\Theta_{a,b,c}) + (a-a')\tau(\Theta_{a',b,c}) \; \hbox{ where } \; a_1+a_2 = a'$, 
\item[$(2).$] $\tau(\Theta_{a,b,c,d}^1(a_1))  =  d\,\tau(\Theta_{a,b,c}) + a_1\,\tau(\Theta_{a - a_1,b,c}) $,
\item[$(3).$] $\tau(\Theta_{a,b,c,d}^2(a_1,b_1))  =  d\,\tau(\Theta_{a,b,c}) + c(a_1 + b_1)(a_2+b_2) + a_1a_2b + b_1b_2a \; \hbox{ where } \; a_2 = a-a_1 \; \hbox{ and } \; b_2 = b - b_1.$
\end{enumerate}

\end{lemma}

We use the identities presented in Lemma \ref{LM0} in order to derive some lower bounds for the number of spanning trees for the graphs $\Theta^0$, $\Theta^1$ and $\Theta^2.$

\begin{lemma}\label{LM1}
We have $$\tau(\Theta_{a,b,c,d}^0(a_1,a_2)) \geq \left\{ 
\begin{array}{ll} 
		(d+\frac{1}{2})\tau(\Theta_{a,b,c}) & \mbox{if } a = 3 \hbox{\;or\;} d \geq 2,  \\ 
		(d+1)\,\tau(\Theta_{a,b,c}) & \mbox{if } a \geq 4 \hbox{\;and\;} d = 1.
\end{array} \right.$$
\end{lemma}

\begin{proof}
Let $\rho(x) = (a-x)(bc+x(b+c))$. By Lemma \ref{LM0}.(1) we have to show

$$\rho(x) \geq   \left\{
\begin{array}{ll} 
		\frac{1}{2}\tau(\Theta_{a,b,c}) & \mbox{if } a = 3 \hbox{\;or\;} d \geq 2,  \\ 
		\tau(\Theta_{a,b,c}) & \mbox{if } a \geq 4 \hbox{\;and\;} d = 1.
\end{array} \right.$$
where $x = a_1 + a_2$ As $\rho$ is a quadratic concave function it is enough to verify the claim for $x \in \{2, a-1\}$ if $d \geq 2$ and $x \in \{2,a-2\}$ if $d = 1$.
Suppose first that $x = 2$, i.e. $a_1 = a_2 = 1$. If $a = 3$ then $\rho(2) = bc+2b+2c \geq \frac{1}{2}(3b+3c+bc) = \frac{1}{2}\tau(\Theta_{a,b,c})$ and if $a \geq 4$, then
\begin{eqnarray}
\rho(2) & = & abc + 2ab + 2ac - 2bc - 4b - 4c \nonumber \\ 
& \geq & (bc+ab+ac)+ (3bc-2bc+ab+ac-4b-4c)  \nonumber \\
& \geq & ab+ac+bc \nonumber \\
& = & \tau(\Theta_{a,b,c}). \nonumber
\end{eqnarray}

Suppose now that $d \geq 2$ and $x=a-1$. Then, we have $\rho(a-1) = bc + (a-1)(b+c) \geq \frac{1}{2}(ab+ac+bc)$ as $a \geq 3.$ Finally, assume $d = 1$ and hence $a \geq 4.$ Then $\rho(a-2) = 2bc+2(a-2)b+2(a-2)c \geq ab+bc+ac$ as $2(a-2) \geq a.$
\end{proof}

\begin{lemma}\label{LM2}
$$\tau(\Theta_{a,b,c,d}^1)(a_1) \geq (d+\frac{1}{2})\tau(\Theta_{a,b,c}).$$
\end{lemma}

\begin{proof}
By Lemma \ref{LM0}.(2) it is enough to show $\rho(x) \geq \frac{1}{2}(ab+ac+bc)$ for each $x \in \left[1,a\right]$, where $\rho(x) = x(bc+b(a-x) + c(a-x))$. As $\rho$ is a quadratic concave function, it is enough to show this inequality only for $x = 1$ and $x = a$.

For $x = 1$, this inequality reduces to $ab+bc+ac \geq 2c+2b$ which trivially holds for $a \geq 2$. 

For $x = a$, we have to show the inequality $abc \geq \frac{1}{2}(ab+bc+ac)$. Notice that $a \geq 2$ and at least one of $b$,$c$ is $\geq 2$, say $b$. Then $ab \geq a+b$ and hence: $$2abc \geq (a+b)c+abc \geq ab+bc+ac.$$
\end{proof}

\begin{lemma}\label{LM3} 
We have $$\tau(\Theta_{a,b,c,d}^2)(a_1,b_1) \geq (d+1)\,\tau(\Theta_{a,b,c}).$$
\end{lemma}

\begin{proof} 
By Lemma \ref{LM0}.(3), we have to show that $$c(a_1+b_1)(a_2+b_2) + a_1a_2b + b_1b_2a \geq ac+bc+ab.$$ It will be enough if we show that: $$(a_1+b_1)(a_2+b_2) \geq a+b \quad \hbox{\; and \;} \quad a_1a_2b+b_1b_2a \geq ab.$$
The first inequality follows immediately by the known fact $xy \geq x+y$ for $x,y \geq 2$ (just set $x = a_1+b_1$ and $y = a_2+b_2$).

For the second inequality, it is enough to observe that $a_1a_2 \geq \frac{a}{2}$ and $b_1b_2 \geq \frac{b}{2}$ as $xy \geq \frac{x+y}{2}$ whenever $x,y \geq 1$.
\end{proof}
Using the derived bounds from Lemmas \ref{LM1}.--\ref{LM3}. we can state the following Corollary that turns out to be useful in the next section. 

\begin{cor}\label{Cor}
Let $G$ be the graph obtained after interconnecting two vertices of $\Theta_{a,b,c}$ with a $d$-path $(d\geq 1).$ Then
$$
\tau(G) \geq \left\{
									\begin{array}{ll}
									3\tau(\Theta_{a,b,c})/2 & d = 1, \\
									5\tau(\Theta_{a,b,c})/2 & d = 2, \\
									d\tau(\Theta_{a,b,c}) & d \geq 3. \\ 
									\end{array}
\right.	
$$
\end{cor}
\begin{proof}
Let $G$ be the graph obtained after interconnecting two vertices $x,y$ of $\Theta_{a,b,c}$ with a $d$-path. If $d = 1$ or $d = 2$ then $x$ and $y$ are distinct and we see from Lemmas \ref{LM1}, \ref{LM2} and \ref{LM3} that the stated inequality holds. If $d \geq 3$ and $x$, $y$ are distinct it follows from Lemmas \ref{LM1}.--\ref{LM3}. that the number of spanning trees of $G$ is at least $\frac{7}{2}$ times greater than $\tau(\Theta_{a,b,c})$. However, if $x = y$ then we obtain a graph that has $C_d$ and $\Theta_{a,b,c}$ as blocks and thus $\tau(G) = d\,\tau(\Theta_{a,b,c}),$ which implies our claim.
\end{proof}

\section{Inequality with functions $\alpha$ and $\beta$}
This section presents our main result. We derive some bounds and equalities for $\alpha$ and $\beta$ for some small values and afterwards we show that for $n \not \equiv 2 \pmod{3}$
$$ \beta(n) \leq \frac{n+13}{4}, $$ 
except for a few cases.

\begin{prop}\label{FIXEDPOINT}
$\alpha(22) = \beta(22) = 22.$ 
\end{prop}

\begin{proof}
It is enough to show $\alpha(22) = 22.$ Let us assume $G$ is a graph on $|V(G)| < 22$ vertices with precisely $22$ spanning trees. Since $22$ is not expressible as $ab+ac+bc$ for $a \geq 1$, $b,c \geq 2$ and since it is not expressible as $22 = ab$ for $a,b \geq 3$ it follows that there exist integers $a \geq 1$ and $b,c \geq 2$ such that $\Theta_{a,b,c} \subset G$. The only $\Theta$'s graphs having less than 22 spanning trees and vertices are $\Theta_{1,2,2}$, $\Theta_{1,2,3}$, $\Theta_{2,2,2}$, $\Theta_{1,2,4}$, $\Theta_{1,3,3}$, $\Theta_{2,3,3}$, $\Theta_{2,2,3}$, $\Theta_{1,2,5}$, $\Theta_{1,3,4}$, $\Theta_{1,2,6}$ and $\Theta_{2,2,4}$, so $G$ contains at least one of them.  Moreover, we see from Corollary \ref{Cor} that introducing a $d$-path $(d \geq 1)$ to any of the graphs $\Theta_{1,3,3}$, $\Theta_{2,3,3}$, $\Theta_{2,2,3}$, $\Theta_{1,2,5}$, $\Theta_{1,3,4}$, $\Theta_{1,2,6}$ or $\Theta_{2,2,4}$ yields a graph with at least $\lceil \frac{15\cdot3}{2} \rceil = 23$ spanning trees.  Thus we can conclude that at least one of $\Theta_{1,2,2}$, $\Theta_{1,2,3}$, $\Theta_{2,2,2}$ or $\Theta_{1,2,4}$ is a proper subgraph of $G$. 

Assume $\Theta_{1,2,2} \subset G.$ Since $\Theta_{1,2,2}$ has four vertices and since any graph on four vertices has at most 16 spanning trees it follows that there exists a subgraph $F$ of $G$ that is obtained after interconnecting two vertices of $\Theta_{1,2,2}$ with some $k$-path $(k>1).$ If $k \geq 3$ then from Corollary \ref{Cor} we have that $F$ is a graph with at least $8\cdot3 > 22$ spanning trees, so $k = 2.$ Observe that this implies that we can assume  $\Theta_{2,2,2} \subset G$, $\Theta_{1,2,3} \subset G$ or $\Theta_{1,2,4} \subset G$. Moreover, using Corollary \ref{Cor} we see that by adding a $d$-path ($d \geq 2$) to any of the graphs $\Theta_{1,2,3}$, $\Theta_{2,2,2}$, $\Theta_{1,2,4}$ one obtains a graph with at least $\lceil \frac{11\cdot5}{2} \rceil = 28$ spanning trees. Combining these two facts we conclude that $G' + e \subseteq G$ for an edge $e$ and $G' \in \{ \Theta_{1,2,3}, \Theta_{2,2,2}, \Theta_{1,2,4} \}.$ 

After consulting Table \ref{Tb2} we see that all such graphs $G'+e$ have more than 22 spanning trees with the exception of $\Theta^1_{2,2,2,1}(2)$ and $\Theta^1_{3,2,1,1}(2)$ which have 20 and 21 spanning trees, respectively. We now deduce from inequality (\ref{inequ}) that $H = \Theta^1_{2,2,2,1}(2) + e' \subseteq G$ for an edge $e'$ since the addition of an edge or a path to $\Theta^1_{3,2,1,1}(2)$ produces a graph with at least 23 spanning trees while introducing a $\geq$\! 2-path to $\Theta_{2,2,2,1}^1(2)$ yields a graph having at least 40 spanning trees. Using the recurrence (\ref{rec}) we see that $$\tau(G) \geq \tau(H-e') + \tau(H/e') > \tau(\Theta_{2,2,2,1}^1(2)) + \tau(\Theta_{1,2,2}) = 28$$ since $\Theta_{1,2,2} \subset H/e'$ regardless of the choice of $e'.$ This implies that there is no graph having 22 spanning trees and less than 22 vertices, thus proving the stated lemma.

\end{proof}

\begin{table}[ht]
\begin{center}
\begin{tabular}{|c|l|c|c|}
\hline
$G$ & $G+e$ & $\tau(G+e)$ \\
\hline
$\Theta_{1,2,3}$	&		$\Theta_{3,2,1,1}^1(2)$  & 21 \\
&		$\Theta_{3,2,1,1}^2(1,1)$  & 24 \\
\hline
$\Theta_{1,2,4}$	&		$\Theta_{4,2,1,1}^1(2)$  & 30 \\
&		$\Theta_{4,2,1,1}^2(1,1)$  & 32 \\
&		$\Theta_{4,2,1,1}^2(2,1)$  & 35 \\
&		$\Theta_{4,2,1,1}^0(1,1)$  & 30 \\
\hline
$\Theta_{1,3,3}$	&		$\Theta_{3,3,1,1}^1(2)$  & 29 \\
&		$\Theta_{3,3,1,1}^2(1,1)$  & 35 \\
&		$\Theta_{3,3,1,1}^2(1,2)$  & 36 \\
\hline
$\Theta_{2,2,2}$ &		$\Theta_{2,2,2,1}^2(1,1)$  & 24 \\
&		$\Theta_{2,2,2,1}^1(2)$  & 20 \\
\hline

$\Theta_{2,2,3}$ &		$\Theta_{3,2,2,1}^1(2)$  & 32\\
&		$\Theta_{3,2,2,1}^2(1,1)$  & 35\\
&		$\Theta_{2,2,3,1}^2(1,1)$  & 32 \\
\hline
\end{tabular}
\caption{Graphs constructed in the proof of Proposition \ref{prop1} and Theorem \ref{Thm2}.}\label{Tb2}
\end{center}
\end{table}

\begin{prop}\label{prop1}
The relations $\beta(9) = 6$, $\alpha(18) = 8$, $\alpha(25) > 9$, $\beta(37) \leq 9$, $\beta(58) \leq 10$, $\beta(30) \leq 8$ hold.
\end{prop}

\begin{proof}
We consider each claim individually:
\begin{itemize}
\item{$\beta(9) = 6.$} It is easy to verify that there is no graph on less thna 6 edges with 9 spanning trees and that $\tau(C_{3,3}) = 9.$

\item{$\beta(30) \leq 8.$} We see from the first equality in Lemma \ref{LM0} that $\Theta_{4,1,2,1}^0(1,1)$ has 30 spanning trees and 8 edges from where the stated inequality follows.

\item{$\beta(37) \geq 9.$} According to the second identity stated in Lemma \ref{LM0}, $\tau(\Theta_{3,1,4,1}^1(2)) = 37.$ This implies $\beta(37) \leq 9$ since $\Theta_{3,1,4,1}^1(2)(1,1)$ has 9 edges.

\item{$\beta(58) \leq 10.$} From the second identity derived in Lemma \ref{LM0} we see that $\Theta_{4,3,2,1}^1(2)$ is a graph with 10 edges having precisely 58 spanning trees, which implies $\beta(58) \leq 10.$

\item{$\alpha(25) = 9.$} Clearly, $C_{5,5}$ is a graph of order 9 with $25$ spanning trees. Assume now $\alpha(25) \leq 8$ and let $G$ be a graph with 25 spanning trees having less than 9 vertices. The only graph of the form $C_{a,b}$ which has 25 spanning trees is $C_{5,5}$ but $|V(C_{5,5})| = 9.$ Moreover, there exist no integers $a \geq 1$, $b,c \geq 2$ such that $|V(\Theta_{a,b,c})| < 9$ and $\tau(\Theta_{a,b,c}) = 25$. Thus, it follows that $G$ contains as a subgraph some $\Theta_{a,b,c}$ with at most 8 vertices, i.e $a+b+c \leq 9$. It can be verified that adding a $d$-path $(d \geq 1)$ to any of such graphs produces a graph with at least 25 spanning trees with the exception of the graphs $\Theta_{1,2,2}$, $\Theta_{1,2,3}$, $\Theta_{1,2,4}$, $\Theta_{1,3,3}$, $\Theta_{2,2,2}$, $\Theta_{2,2,3}$. Moreover, we see from Corollary \ref{Cor} that adding a $2$-path to the graph $\Theta_{a,b,c}$ yields a graph with at least $\frac{5}{2}\tau(\Theta_{a,b,c})$ spanning trees. We now split the proof into two cases:

\begin{itemize}
\item[\textbf{Case 1:}] {\em $\Theta_{1,2,3}$, $\Theta_{1,2,4}$, $\Theta_{1,3,3}$,$\Theta_{2,2,2}$ or $\Theta_{2,2,3}$ is a subgraph of $G.$} By the above observation, adding a $2$-path to any of the graphs covered in this case produces a graph with at least $\frac{5}{2}\cdot11 > 25$ spanning trees. Thus $G'+e \subseteq G$ for $G' \in \{ \Theta_{1,2,3}, \Theta_{1,2,4}, \Theta_{1,3,3}, \Theta_{2,2,2}, \Theta_{2,2,3} \}$ and an edge $e$. Table \ref{Tb2} lists all the graphs (up to isomorphism) and the respective number of spanning trees that can be constructed after adding an edge to $G'.$ The fact that adding an edge or a path to a connected graph increases its number of spanning trees by at least 2 reduces our choice for $G'$ to the elements of the set $\{ \Theta^1_{2,2,2,1}(2), \Theta^1_{3,2,1,1}(2)\}.$ Moreover, we saw in the proof of Proposition \ref{FIXEDPOINT} that $\tau(\Theta_{2,2,2,1}^1(2)+e') > 28$ for any edge $e'.$ As we cannot add a $\geq$ 2-path to $\Theta^1_{2,2,2,1}(2)$ without obtaining a graph with less than 40 spanning trees (by the second inequality stated in (\ref{inequ})) it follows that $ G \subseteq H = \Theta_{3,2,1,1}^1(2)+e'$ for an edge $e'$. Using the deletion-contraction recurrence stated in (\ref{rec}), we now see that $$\tau(G) \geq \tau(H-e')+\tau(H/e') > \tau(\Theta_{3,2,1,1}^1(2)) + \tau(\Theta_{1,2,2}),$$ since the graph $H/e'$ contains $\Theta_{1,2,2}$ as a subgraph independently of the choice of $e'$.

\item[\textbf{Case 2:}] {\em $\Theta_{1,2,2}$ is a subgraph of $G.$} Observe that $\Theta_{1,2,2}+e = K_4$ and $\tau(K_4) = 16$. Moreover, we see from Lemmas \ref{LM1}.--\ref{LM3}. that adding a 3-path to $\Theta_{1,2,2}$ produces a graph with at least $(3+\frac{1}{2})\cdot 8 > 25$ spanning trees unless the path interconnects the same vertex in which case we obtain a graph $H$ with $3\cdot8 = 24$ spanning trees. Observe that a similar argument holds when we add a longer path to $\Theta_{1,2,2}$. Since by virtue of inequality (\ref{inequ}) we cannot introduce an edge or a path to $H$ that would produce a graph with less than $26$ spanning trees we conclude that
$G = \Theta^i_{1,2,2,2}$ for $i\in\{0,1,2\}$. But this implies that $\Theta_{2,2,2} \subset G$ or $\Theta_{1,2,3} \subset G$ and we can use the same reasoning as in Case 1 to conclude that there is no graph $G$ satisifying the stated properties.
\end{itemize}

The proof of the stated inequality is now complete since the above two cases show that there exist no graph on less than $9$ vertices that has $25$ spanning trees.

\item{$\alpha(18) = 8.$} Since $\tau(C_{3,6}) = 18$, it follows that $\alpha(18) \leq 8.$ To prove that $\alpha(18) \geq 8$ and thus $\alpha(18) = 8.$ One can now use an argument similar to the previous case $\alpha(25) = 9.$

\end{itemize}
\end{proof}

\begin{prop}\label{Prop2}
If \begingroup \catcode`,\active \def,{\char`\,\penalty0 } $n \in \{40, 42, 45, 48, 60, 70, 72, 78, 85, 88, 102, 105, 112, 120, 130, 133, 165, 168, 190, 210, 232, 240, 253, 273, 280, 312, 330, 345, 357, 385, 408, 462, 520, 760, 840, 1320, 1365, 1848\} \cup I^*$ \endgroup then $\beta(n) \leq \frac{n+13}{4}.$
\end{prop}

\begin{proof}
Let  $n$ be a number from the set defined in the statement of this lemma. If $n$ is of the form $n = 5k$ for $k \geq 7$ then $|E(\Theta_{5,k})| \leq \frac{n+13}{4}$ since $5+k \leq \frac{5k+13}{4}$ for every $k \geq 7.$ Therefore, the inequality holds for every $n$ divisible by 5 since every such $n$ that is in the list satisfies the required condition. The same reasoning can be applied to the cases when $n$ is a divisor of 6 or a divisor of 7 leaving us to verify the inequalities for the numbers $n \in \{88,232,253\}\cup I^*.$ It is easy to verify that for $n \in \{88,232,253\}$ the graphs $C_{8,11}$, $C_{8,29}$, $C_{11,23}$ have $88$, $232$ and $253$ spanning trees, respectively, and that the inequalities on the number of edges and vertices are satisifed.

If $n \in I^*$ then we use the fact that $n$ has at least three odd prime factors. Let $p_1,p_2$ be two of these prime factors of $n$ and let $d = \frac{n}{p_1p_2}$. We construct the graph $C$ by taking disjoint cycles $C_{p_1}$, $C_{p_2}$, $C_{d}$ and identifying a vertex of every cycle. Observe that the graph is well defined as $d \geq 3$. The obtained graph clearly has $n$ spanning trees and $$p_1+p_2+d \leq 3+3+\frac{n}{9} \leq \frac{n+7}{4}\;,$$ edges. The last inequality following from the fact that $n $ is clearly greater than 30.
\end{proof}

\begin{theorem}\label{Thm2}
Let $n\geq 3$ be an integer. Then, $\beta(n) \leq \frac{n+7}{3} $ if and only if \begingroup \catcode`,\active \def,{\char`\,\penalty0 } $n \notin \{3,4,5,6,7,9,10,13,18,22\}.$ \endgroup Moreover, if $n \not\equiv 2 \pmod{3}$ and $n \ne 25$ then $\beta(n) \leq \frac{n+13}{4}.$
\end{theorem}

\begin{proof}
If $n$ is not idoneal then $n$ is expressible as $n = ab + ac + bc$ for some integers $0 < a < b < c$. The reader may verify that the graph $\Theta_{a,b,c}$ has $a+b+c$ edges and precisely $n$ spanning trees. Therefore for every non idoneal $n$, $\beta(n) \leq a+b+c.$ Observe also that at most one of $a$, $b$, $c$ is 1, therefore $a+b+c$ is maximal when $a = 1$, $b = 2$ and $c = \frac{n-2}{3}.$ The latter also implies that $n \equiv 2 \pmod{3}.$ Otherwise, if $n \not \equiv 2 \pmod{3},$ then the maximum for $a+b+c$ is attained whenever $a = 2$, $b = 2$ and $c = \frac{n-3}{4}.$ Summing up the resulting equalities we conclucde that $a+b+c \leq \frac{n+4}{3}$ if $n \equiv 2 \pmod{3}$ and $a+b+c \leq \frac{n+9}{4}$ otherwise.

We now consider the case when $n$ is an idoneal number from the set \begingroup \catcode`,\active \def,{\char`\,\penalty0 } $I \setminus \{3,4,5,6,7,10,13,22 \} \cup I^*.$ \endgroup If $n$ is from the set\begingroup \catcode`,\active \def,{\char`\,\penalty0 }  $n \in \{40, 42, 45, 48, 60, 70, 72, 78, 85, 88, 102, 105, 112, 120, 130, 133, 165, 168, 190, 210, 232, 240, 253, 273, 280, 312, 330, 345, 357, 385, 408, 462, 520, 760, 840, 1320, 1365, 1848\}$ \endgroup then the inequality immediately follows from Proposition \ref{Prop2}; moreover, if $n$ is from $\{ 8, 12, 15, 16, 21, 24, 28, 33, 57, 93,177\}$ then it can be verified that the graphs $\Theta_{2,2,1}$, $\Theta_{2,2,2}$, $\Theta_{3,3,1}$, $\Theta_{2,2,3}$, $\Theta_{3,3,2}$, $\Theta_{2,2,5}$, $\Theta_{2,2,6}$, $\Theta_{3,3,4}$, $\Theta_{3,3,8}$, $\Theta_{3,3,14}$, $\Theta_{3,3,28}$ have 8, 12, 15, 16, 21, 24, 28, 33, 57, 93 and 177 spanning trees respectively while also satisfying the stated inequality for the number of edges.

We are now left with the idoneal numbers 9, 18, 25, 30, 37 and 58. We see from Proposition \ref{prop1} that for $n \in \{30,37,58\}$ the inequality holds and that for $n \in \{9,18,25\}$ the inequality does not hold, which now proves our theorem.

\end{proof}

Generalizing the graph $\Theta_{a,b,c}$ to the graph containing four disjoint paths of lengths $a,b,c,d$ that interconnect two vertices, one obtains a graph with $abc+abd+acd+bcd$ spanning trees. Using this generalization for a higher number of disjoint paths one could lower the constant factor in our inequality to an arbitrary amount. The fact that there is no result known to the authors related to the solvability of the respective equations and the fact that there could exist a superior construction yielding a better lower bound, motivates the following question:

\begin{Que}
Given a real number $c > 0$ is there an integer $n_0$ such that for every $n > n_0$ we have $$\alpha(n) < cn?$$
\end{Que}

\section*{Acknowledgments}

The authors are thankful for the numerous corrections suggested by the anonymous referees.

\end{document}